\documentclass[final,3p,times]{elsarticle}
\usepackage{amsmath,amssymb}
\usepackage[all]{xy}
\usepackage{latexsym}
\usepackage{amsthm,color}
\usepackage{amsmath,amscd,verbatim}
\usepackage{hyperref}
\usepackage{graphicx}

\input diagxy

\theoremstyle{plain}
  \newtheorem{thm}{Theorem}[section]
  \newtheorem{lem}[thm]{Lemma}
  \newtheorem{prop}[thm]{Proposition}
  \newtheorem{cor}[thm]{Corollary}
\theoremstyle{definition}
  \newtheorem{defn}[thm]{Definition}
  
  \newtheorem{exmp}[thm]{Example}
  \newtheorem{rem}[thm]{Remark}

\makeatletter \def\ps@pprintTitle{  \let\@oddhead\@empty  \let\@evenhead\@empty  \def\@oddfoot{\centerline{\thepage}} \let\@evenfoot\@oddfoot} \makeatother

\begin{document}

\newcommand{\oto}{{\to\hspace*{-3.1ex}{\circ}\hspace*{1.9ex}}}
\newcommand{\lam}{\lambda}
\newcommand{\da}{\downarrow}
\newcommand{\Da}{\Downarrow\!}
\newcommand{\D}{\Delta}
\newcommand{\ua}{\uparrow}
\newcommand{\ra}{\rightarrow}
\newcommand{\la}{\leftarrow}
\newcommand{\lra}{\longrightarrow}
\newcommand{\lla}{\longleftarrow}
\newcommand{\rat}{\!\rightarrowtail\!}
\newcommand{\up}{\upsilon}
\newcommand{\Up}{\Upsilon}
\newcommand{\ep}{\epsilon}
\newcommand{\ga}{\gamma}
\newcommand{\Ga}{\Gamma}
\newcommand{\Lam}{\Lambda}
\newcommand{\CF}{{\cal F}}
\newcommand{\CA}{{\mathcal{A}}}
\newcommand{\CG}{{\cal G}}
\newcommand{\CH}{{\cal H}}
\newcommand{\CN}{{\mathcal{N}}}
\newcommand{\CB}{{\cal B}}
\newcommand{\CT}{{\cal T}}
\newcommand{\CS}{{\cal S}}
\newcommand{\CP}{{\cal P}}
\newcommand{\CU}{\mathcal{U}}
\newcommand{\CW}{\mathcal{W}}
\newcommand{\CQ}{\mathcal{Q}}
\newcommand{\mq}{\mathcal{Q}}
\newcommand{\cu}{{\underline{\cup}}}
\newcommand{\ca}{{\underline{\cap}}}
\newcommand{\nb}{{\rm int}}
\newcommand{\Si}{\Sigma}
\newcommand{\si}{\sigma}
\newcommand{\Om}{\Omega}
\newcommand{\bm}{\bibitem}
\newcommand{\bv}{\bigvee}
\newcommand{\bw}{\bigwedge}
\newcommand{\dda}{\downdownarrows}
\newcommand{\dia}{\diamondsuit}
\newcommand{\y}{{\bf y}}
\newcommand{\colim}{{\rm colim}}
\newcommand{\fR}{R^{\!\forall}}
\newcommand{\eR}{R_{\!\exists}}
\newcommand{\dR}{R^{\!\da}}
\newcommand{\uR}{R_{\!\ua}}
\newcommand{\swa}{{\swarrow}}
\newcommand{\sea}{{\searrow}}
\newcommand{\bbA}{{\mathbb{A}}}
\newcommand{\frX}{{\mathfrak{X}}}
\newcommand{\frx}{{\mathfrak{x}}}
\newcommand{\frY}{{\mathfrak{Y}}}
\newcommand{\fry}{{\mathfrak{y}}}
\newcommand{\frZ}{{\mathfrak{Z}}}
\newcommand{\frz}{{\mathfrak{z}}}

\newcommand{\id}{{\rm id}}
\newcommand{\bbU}{{\mathbb{U}}}
\newcommand{\bbP}{{\mathbb{P}}}
\newcommand{\bbT}{{\mathbb{T}}}
\newcommand{\bbS}{{\mathbb{S}}}
\newcommand{\CV}{{\mathcal{V}}}
\newcommand{\sU}{{\sf{U}}}
\newcommand{\sV}{{\sf{V}}}
\newcommand{\sW}{{\sf{W}}}
\newcommand{\sP}{{\sf P}}
\newcommand{\sy}{{\sf{y}}}
\newcommand{\sk}{{\sf{k}}}
\newcommand{\sfs}{{\sf{s}}}
\newcommand{\h}{\text{-}}

\numberwithin{equation}{section}
\renewcommand{\theequation}{\thesection.\arabic{equation}}

\begin{frontmatter}
\title{Quantale-valued Approach Spaces via Closure and Convergence}

\author[S]{Hongliang Lai\fnref{A}}
\ead{hllai@scu.edu.cn}

\author[Y]{Walter Tholen\corref{cor}\fnref{A}}
\ead{tholen@mathstat.yorku.ca}

\address[S]{School of Mathematics, Sichuan University, Chengdu 610064, China}
\address[Y]{Department of Mathematics and Statistics, York University, Toronto, Ontario, Canada, M3J 1P3}

\cortext[cor]{Corresponding author.}

\fntext[A]{Partial financial assistance by National Natural Science Foundation of China (11101297), International Visiting Program for Excellent Young Scholars of Sichuan University, and by the Natural Sciences and Engineering Research Council (NSERC) of Canada is gratefully acknowledged. This work was completed while the first author held a Visiting Professorship at York University.}

\begin{abstract}

For a quantale $\sV$ we introduce $\sV$-approach spaces via $\sV$-valued point-set-distance functions and, when $\sV$ is completely distributive, characterize them in terms of both, so-called closure towers and ultrafilter convergence relations. When $\sV$ is the two-element chain ${\sf 2}$, the extended real half-line $[0,\infty]$, or the quantale ${\bf{\Delta}}$ of distance distribution functions, the general setting produces known and new results on topological spaces, approach spaces, and the only recently considered probabilistic approach spaces, as well as on their functorial interactions with each other.
\end{abstract}


\end{frontmatter}

\section{Introduction}

Lowen's \cite{Lowen} approach spaces provide an ideal synthesis of Lawvere's \cite{Lawvere} presentation of metric spaces (as small $[0,\infty]$-enriched categories) and the Manes-Barr \cite{Manes, Barr} presentation of topological spaces in terms of ultrafilter convergence, as demonstrated first in \cite{CleHof2003}; see also \cite{MonTop}. Several authors have investigated {\em probabilistic} generalizations of these concepts (see in particular \cite{BrockKent, HofmannReis, Jager}), which suggests that a general quantale-based study of approach spaces should be developed, in order to treat these and other new concepts efficiently in a unified manner, in terms of both, ``distance" or ``closure",  and ``convergence". In this paper we provide such a treatment, working with an arbitrary quantale $\sV=(\sV,\otimes,\sk)$
which, for the main results of the paper, is required to be completely distributive. For $\sV={\sf 2}$ the two-element chain, our results reproduce the equivalence of the descriptions of topologies in terms of closure and ultrafilter convergence; for $\sV=[0,\infty]$ (ordered by the natural $\geq$ and structured by $+$ as the quantalic $\otimes$), one obtains the known equivalent descriptions of approach spaces in terms of point-set distances and of ultrafilter convergence; for
 $\sV={\bf{\Delta}}$ the quantale of {\em distance distribution functions} $\varphi: [0,\infty]\to[0,1]$, required to satisfy the left-continuity condition $\varphi(\beta)={\rm sup}_{\alpha<\beta}\varphi(\alpha)$ for all $\beta\in [0,\infty]$, the corresponding equivalence is established here also for {\em probabilistic approach spaces}.  A major advantage of working in the harmonized context of a general quantale is that it actually makes the proofs more transparent to us than if they were carried out in the concrete quantales that we are interested in.

While this paper is built on the methods of {\em monoidal topology} as developed in \cite{ClementinoTholen2003, CleHofTho, MonTop} and elsewhere (see in particular \cite{Hohle}), in this paper we emphasize the lax-algebraic setting presented in \cite{LaxDistLaws}, which is summarized in this paper to the extent needed. This setting is in fact well motivated by Lowen's original axioms for an approach space $(X,\delta)$ in terms of its point-set distance function $\delta: X\times \sP X\to [0,\infty]$, listed in \cite{Lowen} with $\sP X={\sf 2}^X$, as follows:
\begin{enumerate}[(D1)]
\item $\forall x\in X: \;\delta(x,\{x\})=0$,
\item $\forall x \in X: \;\delta(x,\emptyset)=\infty$,
\item $\forall x \in X,\; A, B\subseteq X:\;\delta(x,A\cup B)={\rm min}\{\delta(x,A),\delta(x,B)\}$,
\item $\forall x\in X, \; A \subseteq X,\; \varepsilon \in [0,\infty]:\;\delta(x,A)\leq \delta(x,A^{(\varepsilon)})+ \varepsilon,
\text{ where } A^{(\varepsilon)}:=\{x\in X\;|\;\delta(x,A)\leq\varepsilon\}$.
\end{enumerate}
Since (D2), (D3) require $\delta(x,-): (\sP X, \subseteq)\to([0,\infty],\geq)$ to preserve finite joins for every fixed $x\in X$, we are led to describe $\delta$ equivalently as a function
$$c:\sP X\to[0,\infty]^X\quad\quad\quad(*)$$
-- which also avoids the quantification over $x$ in each of these axioms. Now (D1) and (D4) may be interpreted as the reflexivity and transitivity axioms for a lax $(\bbP,[0,\infty])$-algebra in the sense of \cite{MonTop}, where $\bbP$ is the powerset monad of {\bf Set}, suitably extended to $[0,\infty]$-valued relations of sets. Equivalently, as we will show in this paper, (D1) and (D4) provide $X$ with a $[0,\infty]$-indexed {\em closure tower} (named so after the  terminology used in \cite{BrockKent, Zhang}), the members of which are collectively extensive, monotone and idempotent, in a sense that we make precise in the general context of a quantale in Proposition \ref{closuretowers}. In this way we obtain new characterizations of approach spaces and of probabilistic approach spaces in terms of closure, which we summarize at the end of Section 2.

Lowen \cite{Lowen} also gave the equivalent description of the structure of an approach space $X$  in terms of a {\em limit operator} ${\sf F}X\to [0,\infty]^X$, which assigns to every filter on $X$ a function that provides for every $x\in X$ a measure of ``how far away $x$ is from being a limit point" of the given filter. As first shown in \cite{CleHof2003}, it suffices to restrict this operator to ultrafilters, so that the structure may in fact be given by a map
$$\ell:{\sf U}X\to [0,\infty]^X\quad\quad\quad(**)$$
satisfying two axioms that correspond to the reflexivity and transitivity conditions for a lax $({\mathbb U},[0,\infty])$-algebra structure on $X$ as described in \cite{MonTop}, with $\mathbb U$ denoting the ultrafilter monad of {\bf Set}, understood to be laxly extended from maps to $[0,\infty]$-valued relations. 

The presentations $(*), (**)$ motivated the study of {\em lax} $(\lam,\sV)${\em -algebras} in \cite{LaxDistLaws}, {\em i.e.}, of sets provided with a map
$$c:TX\to \sV^X$$
satisfying two basic axioms. Here, for a {\bf Set}-monad $\bbT=(T,m,e)$ and the given quantale
$\sV$, $\lam$ is a {\em lax distributive law} of $\bbT$ over $\bbP_{\sV}$, which links $\bbT$ with $\sV$,
as encoded by the $\sV$-powerset monad $\bbP_{\sV}=(\sP_{\sV},{\sf s},\sy)$, with $\sP_{\sV}X=\sV^X$. For $\bbT=\bbP=\bbP_{\sf 2}$ and a naturally chosen lax distributive law, the corresponding lax algebras are $\sV${\em -closure
spaces}, satisfying the $\sV$-versions of (D1), (D4); they are $\sV${\em -approach spaces} when they also satisfy the $\sV$-versions of (D2), (D3). The main result of the paper (Theorem \ref{mainthm}) describes them equivalently as the lax algebras with respect to a naturally chosen lax distributive law of the ultrafilter monad $\bbU$ over $\bbP_{\sV}$, provided that $\sV$ is completely distributive. The relevant isomorphism of categories comes about as the restriction of an adjunction, the left-adjoint functor of which is an {\em algebraic functor} as discussed in \cite{LaxDistLaws} (in generalization of the well-known concept presented in \cite{CleHofTho, MonTop}).
For $\sV={\bf{\Delta}}$ our general result produces a new characterization of 
probabilistic approach spaces in terms of ultrafilter convergence (Corollary \ref{ProbAppUltra}).

In the last section we study so-called {\em change-of-base functors} (see \cite{CleHofTho, MonTop, LaxDistLaws}) for the categories at issue in this paper. An application of our general result (Theorem \ref{equivalences}) gives a unified proof for the known facts that {\bf Top} may be fully emdedded into {\bf App} as a simultaneously reflective and coreflective subcategory which, in turn is reflectively and coreflectively embedded into {\bf ProbApp}.

\section{$\sV$-approach spaces via closure}
Throughout the paper, let $\sV=(\sV,\otimes,{\sf k})$ be a (unital but not necessarily commutative) {\em quantale, i.e.,} a complete lattice with a monoid structure whose binary operation $\otimes$ preserves suprema in each variable. There are no additional provisions for the  tensor-neutral element $\sk$
vis-\`{a}-vis the bottom and top elements in $\sV$, {\em i.e.,} we  exclude neither the case $\sk=\bot$ (so that $|\sV|=1$), nor $\sk < \top$. The $\sV${\em -powerset functor} $\sP_{\sV}: {\bf Set}\to{\bf Set}$ is given by
$$(f:X\to Y)\mapsto (f_!:\sV^X\to\sV^Y),\;f_!(\sigma)(y)=\bigvee_{x\in f^{-1}y}\sigma(x),$$
for all $\sigma:X\to\sV,\;y\in Y$. The functor $\sP_{\sV}$ carries a monad structure, given by
\[\sy_X:X\to\sV^X,\quad(\sy_Xx)(y)=\left\{
\begin{array}{ll}
\sk & \text{if }y=x\\
\bot & \text{otherwise}
\end{array}
\right\},\]
\[\sfs_X:\sV^{\sV^X}\to\sV^X,\quad(\sfs_X\Sigma)(x)=\bigvee_{\sigma\in\sV^X}\Sigma(\sigma)\otimes\sigma(x),\]
for all $x,y\in X$ and $\Sigma:\sV^X\to\sV$.

Let $\bbT=(T,m,e)$ be any monad on {\bf Set}. A {\em lax distributive law} $\lam$ of $\bbT$ over $\bbP_{\sV}=(\sP_{\sV},\sfs,\sy)$ (see \cite{MonTop, LaxDistLaws}, and \cite{Beck} for its original name giver) is a family of maps $\lam_X:T(\sV^X)\to \sV^{TX}\;(X\in{\bf Set})$ which, when one orders maps to a power of $\sV$ pointwise by the order of $\sV$, must satisfy the following conditions:

 \begin{tabular}{llr}
 &\\
 ${\text{(a)}}\quad\forall f:X\to Y:$& $(Tf)_!\cdot \lam_X\leq\lam_Y\cdot T(f_!)$ &(lax naturality of $\lambda$),\\

 ${\text{(b)}}\quad\forall X:$& $\sy_{TX}\leq \lam_X\cdot T\sy_X$  &
(lax ${\mathbb P}_{\sV}$-unit law),\\
${\text{(c)}}\quad\forall X:$ &  $\sfs_{TX}\cdot (\lam_X)_!\cdot \lam_{\sV^X}\leq\lam_X\cdot T\sfs_X$  &(lax ${\mathbb P}_{\sV}$-multiplication law),\\
${\text{(d)}}\quad\forall X:$
& $(e_X)_!\leq\lam_X\cdot e_{\sV^X}$ &
(lax ${\mathbb T}$-unit law), \\
${\text{(e)}}\quad\forall X:$ & $ (m_X)_!\cdot\lam_{TX}\cdot T\lam_X\leq\lam_X\cdot m_{\sV^X}$  &(lax ${\mathbb T}$-multiplication law),\\
${\text{(f)}}\quad\forall g,h:Z\to\sV^X:$ & $g\leq h\Longrightarrow \lam_X\cdot Tg\leq\lam_X\cdot Th$ &(monotonicity).\\
&\\
\end{tabular}
  
\begin{rem}\label{lax ext}
Although we will make use of it only in the next sextion, let us mention here the fact that lax distributive laws of a {\bf Set}-monad $\bbT =(T,m,e)$ over $\bbP_{\sV}$ correspond bijectively to lax extensions $\hat{T}$ of $\bbT$ to the category
$\sV\text{-}{\bf Rel}$ of sets with $\sV$-valued relations $r:X\nrightarrow Y$ as morphisms, which are equivalently displayed as maps $\overleftarrow{r}:Y\to \sP_{\sV}X$ (see \cite{LaxDistLaws} and Exercise III.1.I in  \cite{MonTop}).
Given $\lam$, the lax functor $\hat{T}:{\sV}\text{-}{\bf Rel}\to{\sV}\text{-}{\bf Rel}$ assigns to $r$ the $\sV$-relation 
$\hat{T}r:TX\nrightarrow TY$ defined by
$$\overleftarrow{\hat{T}r}=\lam_X\cdot T\overleftarrow{r}.$$
Conversely, the lax distributive law $\lam$ associated with $\hat{T}$ is given by
$$\lam_X=\overleftarrow{\hat{T}\epsilon_X},$$
with $\epsilon_X:X\nrightarrow \sP_{\sV}X$ the evaluation $\sV$-relation: $\epsilon_X(x,\sigma)=\sigma(x)$.
\end{rem}

\begin{prop}\label{powerdistributes}
The ordinary powerset monad $\bbP=\bbP_{\sf 2}$ distributes laxly over the $\sV$-powerset monad $\bbP_{\sV}$, via
$$\alpha_X:\sP(\sV^X)\to\sV^{\sP X},\quad(\alpha_XS)(A)=\bigwedge_{x\in A}\bigvee_{\sigma\in S}\sigma(x)\quad(S\subseteq\sV^X,\; A\subseteq X).$$
\end{prop}
\begin{proof}
(a) For all $S\subseteq \sV^X, B\subseteq X$ one has
$$((\sP f)_!\cdot\alpha_X(S))(B)
= \bigvee_{A\subseteq X, f(A)=B}(\alpha_XS)(A)=\bigvee_{A\subseteq X, f(A)=B}\bigwedge_{x\in A}\bigvee_{\sigma \in S}\sigma(x).$$
Lax naturality of $\alpha$ follows since, for every $A\subseteq X$ with $f(A)=B$,
$$
\bw_{x\in A}\bv_{\sigma\in S}\sigma(x)\leq\bw_{y \in B}\bv_{\sigma \in S}\bv_{x\in f^{-1}y}\sigma(x)
=\bw_{y\in B}\bv_{\sigma\in S}(f_!\sigma)(y)
=\alpha_Y(f_!(S))(B)=(\alpha_Y\cdot\sP(f_!))(S)(B).
$$

(b) For all $A,B\subseteq X$,
\[(\alpha_X\cdot \sP\sy_X(B))(A)=\bw_{x\in A}\bv_{y\in B}(\sy_Xy)(x)=
\left\{
\begin{array}{ll}
\top & \text{if }A=\emptyset\\
\sk & \text{if }\emptyset\neq A\subseteq B\\
\bot & \text{otherwise}
\end{array}
\right\}\geq (\sy_{\sP X}B)(A).\]

(c) For all $\bbS\subseteq\sV^{\sV^X},A\subseteq X$,
\begin{align*}
(\sfs_{\sP X}\cdot(\alpha_X)_!\cdot\alpha_{\sV^X}(\bbS))(A) & =\bv_{\tau\in\sV^{\sP X}}((\alpha_X)_!\cdot\alpha_{\sV^X}(\bbS))(\tau)\otimes\tau(A)\\
&=\bv_{\tau\in\sV^{\sP X}}\bv_{S\subseteq\sV^X,\alpha_X(S)=\tau}\alpha_{\sV^X}(\bbS)(S)\otimes\tau(A)\\
&\leq\bv_{S\subseteq\sV^X}\alpha_{\sV^X}(\bbS)(S)\otimes\alpha_X(S)(A)\\
&=\bv_{S\subseteq\sV^X}\Big(\bw_{\tau\in S}\bv_{\Sigma\in\bbS}\Sigma(\tau)\Big)\otimes\Big(\bw_{x\in A}\bv_{\sigma\in S}\sigma(x)\Big)\\
&\leq\bv_{S\subseteq\sV^X}\bw_{x\in A}\bv_{\sigma\in S}\Big(\bw_{\tau\in S}\bv_{\Sigma\in\bbS}\Sigma(\tau)\Big)\otimes\sigma(x)\\
&\leq\bv_{S\in\sV^X}\bw_{x\in A}\bv_{\sigma\in S}\bv_{\Sigma\in\bbS}\Sigma(\sigma)\otimes\sigma(x)\\
&\leq\bw_{x\in A}\bv_{\Sigma\in\bbS}\bv_{\sigma\in\sV^X}\Sigma(\sigma)\otimes\sigma(x)\\
&=(\alpha_X\cdot\sfs_X(\bbS))(A).\\
\end{align*}

(d) With $e_X:X\to \sP X$ denoting the map $x\mapsto \{x\}$, for all $\sigma\in\sV^X, A\subseteq X$ one has
\[(e_X)_!(\sigma)(A)=\bv_{x\in X,\{x\}=A}\sigma(x)=\left\{
\begin{array}{ll}
\sigma(x)& \text{if }(\exists x:A=\{x\})\\
\bot & \text{otherwise}
\end{array}
\right\}
\leq\bw_{x\in A}\sigma(x)=(\alpha_X\cdot e_{\sV^X}(\sigma))(A).\]

(e) With $m_X:\sP\sP X\to \sP X$ denoting the map $\CA\mapsto \bigcup\CA$, for all $\CS\subseteq\sV^X, A\subseteq X$ one has
\begin{align*}
((m_X)_! \cdot\alpha_{\sP X}\cdot\sP\alpha_X(\CS))(A)&=\bv_{\CA\subseteq\sP X,\bigcup\CA=A}(\alpha_{\sP X}\cdot\sP\alpha_X(\CS))(\CA)\\
&=\bv_{\CA\subseteq\sP X,\bigcup\CA=A}\bw_{B\in\CA}\bv_{S\in\CS}(\alpha_XS)(B)\\
&=\bv_{\CA\subseteq\sP X,\bigcup\CA=A}\bw_{B\in\CA}\bv_{S\in\CS}\bw_{y\in B}\bv_{\sigma\in S}\sigma(y),\text{   and}\\
(\alpha_X\cdot m_{\sV^X}(\CS))(A)&=\bw_{x\in A}\bv_{S\in \CS}\bv_{\sigma\in S}\sigma(x).
\end{align*}
But whenever $x\in A=\bigcup\CA$, so that $x\in B_0$ for some $B_0\in\CA$, we have
$$\bw_{B\in\CA}\bv_{S\in\CS}\bw_{y\in B}\bv_{\sigma\in S}\sigma(y)\leq
\bv_{S\in\CS}\bw_{y\in B_0}\bv_{\sigma\in S}\sigma(y)\leq\bv_{S\in \CS}\bv_{\sigma\in S}\sigma(x)$$
and may conclude $((m_X)_! \cdot\alpha_{\sP X}\cdot\sP\alpha_X(\CS))(A)\leq
(\alpha_X\cdot m_{\sV^X}(\CS))(A)$.

(f) From $g\leq h$, for all $C\subseteq Z, A\subseteq X$
one obtains immediately
$$\alpha_X(g(C)))(A)=\bw_{x\in A}\bv_{z\in C}(gz)(x)\leq\bw_{x\in A}\bv_{z\in C}(hz)(x)=\alpha_X(h(C)))(A).$$
\end{proof}

\begin{defn}\label{defnlaxalgebra} (1) (\cite{LaxDistLaws})
Let $\lam$ be a lax distributive law of a {\bf Set}-monad $\bbT$ over $\bbP_{\sV}$. A {\em lax} $(\lam,\sV)${\em -algebra} $(X,c)$ is a set $X$ with a map $c:TX\to\sV^X$ satisfying

 \begin{tabular}{lr}
 &\\
 ${\text{(R)}} \quad\sy_X\leq c\cdot e_X$ &
(lax unit law, reflexivity), \\
${\text{(T)}}\quad  \sfs_X\cdot c_!\cdot\lam_X\cdot Tc\leq c\cdot m_X$ & \quad\quad\quad\quad\quad\quad\quad\quad(lax multiplication law, transitivity).\\
&\\
\end{tabular}

A {\em lax homomorphism} $f:(X,c)\to(Y,d)$ of lax $(\lam,\sV)$-algebras is a map $f:X\to Y$ satisfying

\begin{tabular}{lr}
 &\\
 ${\text{(M)}} \quad f_!\cdot c\leq d\cdot Tf$ &\quad\quad\quad\quad\quad\quad\quad\quad\quad\quad
(lax homomorphism law, monotonicity). \\
&\\
\end{tabular}

The resulting category is denoted by $$(\lam,\sV)\text{-}{\bf Alg}.$$
(2) A $\sV${\em -closure space} $(X,c)$ is a lax $(\alpha,\sV)$-algebra, with $\alpha$ as in Proposition \ref{powerdistributes}; it is a $\sV${\em -approach space} if, in addition, $c:\sP X\to \sV^X$ preserves finite joins:
$$\forall x\in X,\,A,B\subseteq X: (c\emptyset)(x)=\bot\text{ and }c(A\cup B)(x)=(cA)(x)\vee (cB)(x).$$
 A lax $\alpha$-homomorphism of $\sV$-closure spaces is also called a {\em contractive} map. We obtain
the category \[\sV\text{-}{\bf Cls}=(\alpha,\sV)\text{-}{\bf Alg}\] and its full subcategory $\sV\text{-}{\bf App}$.
\end{defn}

\begin{rem}\label{TVCat}
If the lax distributive law $\lam$ is equivalently described as a lax extension $\hat{T}$ of $\bbT$ (see Remark \ref{lax ext}), then
$$(\lam,\sV)\text{-}{\bf Alg}\cong(\bbT,\sV,\hat{T})\text{-}{\bf Cat}$$
is the category of $(\bbT,\sV)$-{\em categories}, as defined in \cite{MonTop}. Under this isomorphism (see \cite{LaxDistLaws}, Prop. 6.8), the $(\lam,\sV)$-structure $c:TX\to\sP_{\sV}X$ corresponds to the $\sV$-relation $a:TX\nrightarrow X$ with $\overleftarrow{a^{\circ}}=c$ 
(where $a^{\circ}: X\nrightarrow TX$ is the converse of $a$), and (R) and (T) now read as
$$\sk\leq a(e_X(x),x)\quad{\rm{and}\quad} \hat{T}a(\frX,\fry)\otimes a(\fry,z)\leq a(m_X\frX,z),$$
for all $\frX\in TTX,\fry\in TX, z\in X$.
\end{rem}

The lax extension $\hat{\sP}:\sV\text{-}{\bf Rel}\to\sV\text{-}{\bf Rel}$ corresponding to $\alpha$ of Proposition \ref{powerdistributes} is, after an easy computation, described by
\[\hat{\sP}r(A,B)=\bw_{y\in B}\bv_{x\in A}r(x,y),\]
for all $\sV$-relations $r:X\nrightarrow Y, \, A,B\subseteq X.$ Consequently, stated elementwise, conditions (R), (T), (M) read for $\lam=\alpha$, as follows:

\begin{lem}\label{closureptw}
A map $c:\sP X\to\sV^X$ makes $X$ a $\sV$-closure space if and, only if, $c$ satisfies

\begin{tabular}{ll}
$\mathrm{(R')}\quad \forall x\in X:$& $\sk\leq c(\{x\})(x),$\\
$\mathrm{(T')}\quad \forall \CA\subseteq \sP X, B\subseteq X, z\in X:$ & $ \Big(\bw_{y\in B}\bv_{A\in \CA}(cA)(y)\Big)\otimes (cB)(z)\leq c(\bigcup\CA)(z).$\\
\end{tabular}

A map $f:X\to Y$ of $\sV$-closure spaces $(X,c), (Y,d)$ is contractive if, and only if,

\begin{tabular}{ll}
$\mathrm{(M')}\quad\forall x\in X, A\subseteq X:$&\quad\quad\quad\; $(cA)(x)\leq(df(A))(fx).$\\
&\\
\end{tabular}
\end{lem}

We can now describe the structure of $\sV$-closure spaces in terms of $\sV$-indexed {\em closure towers}, as follows.

\begin{prop}\label{closuretowers}
{\rm (1)} For a $\sV$-closure space $(X,c)$, with $$c^vA:=\{x\in X\;|\;(cA)(x)\geq v\}\quad(v\in\sV, \;A\subseteq X)$$ one obtains a family of maps
$(c^v:\sP X \to \sP X)_{v\in\sV}$ satisfying
\begin{enumerate}
\item[{\rm (C0)}]
$\text{if  }B\subseteq A, \text{ then  }c^vB\subseteq c^vA,$
\item[{\rm (C1)}] $\text{if  }v\leq\bv_{i\in I}u_i, \text{ then   }\bigcap_{i\in I}c^{u_i}A\subseteq c^vA,
$\item[{\rm (C2)}] $A\subseteq c^{\sk}A,$
\item [{\rm (C3)}] $c^uc^vA\subseteq c^{v\otimes u}A,$
\end{enumerate}
for all $A\subseteq X$ and $ u,v,u_i \in V\; (i\in I)$.

{\rm (2)} Conversely, for any family maps $(c^v:\sP X\to\sP X)_{v\in \sV}$ satisfying the conditions {\rm (C0)--(C3)}, putting $$(cA)(x):=\bv\{v\in\sV\;|\;x\in c^vA\}\quad (A\subseteq X,\;x\in X)$$ makes $(X,c)$ a $\sV$-closure space.

{\rm (3)} The correspondences of {\rm (1), (2)} are inverse to each other. Under this bijection, contractivity of a map $f:X\to Y$ is equivalently described by the {\rm continuity} condition
$$\forall A\subseteq X,\;v\in\sV:\;f(c^vA)\subseteq d^v(f(A)).$$
\end{prop}

\begin{proof}
(1) For (C2), consider $x_0\in A$ and put $\CA:=\{\{x\}\;|\;x\in A\}, B:=\{x_0\}$, to obtain $\sk=\sk\otimes\sk\leq (cA)(x_0)$ from (R'), (T'), {\em i.e.,} $x_0\in c^{\sk}A$. To see (C0), one puts $\CA:=\{A\}$ and obtains for $x\in c^vB$ from $B\subseteq A$ and  (C2), (T')
$$v\leq\sk\otimes(cB)(x)\leq\Big(\bw_{y\in B}(cA)(y)\Big)\otimes(cB)(x)\leq(cA)(x),$$
{\em i.e.,} $x\in c^vA$. (C1) follows trivially from the definition of the closure tower, and for (C3) one puts $\CA:=\{A\},\; B:=c^vA$ to obtain, with (T'), $v\otimes c(c^vA)(x)\leq (cA)(x)$ for all $x\in X$. Hence, for $x\in c^u(c^vA)$ one may conclude $v\otimes u\leq (cA)(x)$, which means $x\in c^{v\otimes u}A.$

(2) (R') follows trivially from (C2). In order to show (T'), putting $$v_y:=\bv\{v\in\sV\;|\;\exists A\in \CA:\;y\in c^vA\},$$
for every $y\in B$ we obtain from (C1) $y \in c^{v_y}A$, for some  $A\in \CA$, and then, with $\tilde{v}:=\bw_{y'\in B}v_{y'}$ and (C0),
$y\in c^{v_y}A\subseteq c^{\tilde{v}}A\subseteq c^{\tilde{v}}(\bigcup\CA)$, so that
$B\subseteq c^{\tilde{v}}(\bigcup\CA)$. Now, for every $u\in\sV$, (C0) and (C3) give
 $c^u(B)\subseteq c^u(c^{\tilde{v}}(\bigcup\CA))\subseteq c^{\tilde{v}\otimes u}(\bigcup \CA)$. Consequently,
 for all $x\in c^uB$, we obtain $\tilde{v}\otimes u\leq c(\bigcup \CA)(x)$ and conclude
 $$\Big(\bw_{y\in B}\bv_{A\in \CA}\bv\{v\in\sV\;|\;y\in c^vA\}\Big)\otimes \bv\{u\in\sV\;|\;x\in c^uB\}=\tilde{v}\otimes \bv\{u\in\sV\;|\;x\in c^uB\}\leq c(\bigcup\CA)(x) ,$$
 as desired.

 (3) Given a $\sV$-closure space structure $c$ on $X$, let $(c^v)_{v\in\sV}$
  be the closure tower as in (1) and denote by $\tilde{c}$ the structure obtained from that tower as in (2). Since trivially
 $x\in c^{(cA)(x)}A$, one easily concludes $(\tilde{c}A)(x)=(cA)(x)$ for all $A\subseteq X, x\in X$. Conversely, starting with a closure tower $(c^v)_{v\in\sV}$, forming the corresponding $\sV$-closure space structure $c$ as in (2) and then its induced closure tower
 $(\tilde{c}^v))_{v\in\sV}$ as in (1), we conclude $$\tilde{c}^vA=\{x\in X\;|\;\bv\{u\in\sV\;|\;x\in c^uA\}\geq v\}\subseteq c^vA$$ for all $A\subseteq X$ from (C1), with the reverse inclusion holding trivially.

 Finally, that (M') implies the given continuity condition follows directly from the definitions. In turn, the continuity condition implies (M') when being exploited for $v:=(cA)(x)$, since then  $x\in c^vA$ and therefore $fx\in d^v(f(A))$,
 which means precisely (M').
\end{proof}

\begin{rem}\label{bottom}
(1) Note that, for a $\sV$-closure space $(X,c)$, one has $c^{\bot}A=X$ for all $A\subseteq X$ (including $A=\emptyset$).
Hence, in Proposition \ref{closuretowers}(2), it suffices to require (C0)--(C3) for all those $u,v,u_i \in\sV \;(i\in I)$ that are greater than $\bot$.

(2) If $c$ and $(c^v)_{v\in\sV}$ correspond to each other as in Proposition \ref{closuretowers}(1),(2), then (C3) may be written equivalently as
\begin{enumerate}
\item [{\rm (C3')}] $v\otimes c(c^vA)(x)\leq(cA)(x),$
\end{enumerate}
for all $x\in X, A\subseteq X, v\in\sV$. Indeed, from (C3) one obtains
$$v\otimes c(c^vA)(x)=v\otimes \bv\{u\in\sV\,|\,x\in c^u(c^vA)\}\leq\bv\{v\otimes u\,|\,u\in\sV, x\in c^{v\otimes u}A\}
\leq\bv\{w\in\sV\,|\,x\in c^wA\}=(cA)(x);$$
conversely, given (C3'), one has
$$ x\in c^u(c^vA)\Longrightarrow (cA)(x)\geq v\otimes c(c^vA)(x)
\geq v\otimes u\Longrightarrow x\in c^{v\otimes u}A.$$

(3) For Lawvere's quantale $[0,\infty]$, ordered by the natural $\geq$ and provided with $\otimes=+$, naturally extended to $\infty$, writing $\delta(x,A)=(cA)(x)$ one sees that condition (C3') coincides with (D4) (see Introduction).

\end{rem}


We are now ready to describe $\sV$-approach spaces in terms of closure towers, provided that $\sV$ is {\em constructively completely distributive (ccd)}. Recall that the complete lattice $\sV$ is ccd if,
and only if,  $v=\bv\{u\in\sV\;|\;u\ll  v\}$ for every $v\in \sV$; here $u\ll v$ (``$u$ {\em totally below} $v$") means
$$\forall D\subseteq V:\;v\leq \bv D\Longrightarrow( \exists \,d\in D:u\leq d).$$
Every completely distributive complete lattice in the ordinary sense is ccd, with the validity of the converse implication being equivalent to the Axiom of Choice (see \cite{Wood, MonTop}).

\begin{thm}\label{approachclosure}
Let $\sV$ be constructively completely distributive. Then a $\sV$-closure space is a $\sV$-approach space
if, and only if, its closure tower $(c^v)_{v\in\sV}$ satisfies
\begin{enumerate}
\item[{\rm(C4)}] $c^v\emptyset=\emptyset$,
\item[{\rm(C5)}] $c^v(A\cup B)=\bigcap_{u\ll v}(c^uA\cup c^uB)$,
\end{enumerate}
for all $v\in\sV, v>\bot$, and $ A,B\subseteq X$.
\end{thm}

\begin{proof}
For the $\sV$-closure space $(X,c)$ to be a $\sV$-approach space means, by definition,
$$(c\emptyset)(x)=\bot\text{  and  }c(A\cup B)(x)=(cA)(x)\vee (cB)(x),$$
for all $A, B\subseteq X,\;x\in X$,  and from Proposition \ref{closuretowers} and Remark \ref{bottom} we recall
 $$ (cA)(x)=\bv\{v\in\sV\;|\; x\in c^vA\}=\bv\{v\in\sV\;|\;v>\bot,x\in c^vA\}.$$
 Trivially then, $(c\emptyset)(x)=\bot$ for all $x\in X$ if, and only if,
$c^v\emptyset=\emptyset$ for all $v>\bot$.

When $\sV$ is completely distributive, from (C1),(C0) one obtains, for all $v\in\sV, A, B \subseteq X,$
$$c^v(A\cup B)=\bigcap_{u\ll v}c^u(A\cup B)\supseteq \bigcap_{u\ll v}(c^uA\cup c^uB).$$
Hence, with the equivalences
\begin{align*}
&\forall x\in X:\;c(A\cup B)(x)\leq(cA)(x)\vee(cB)(x)\\
\iff&\forall x\in X,v\in\sV:\; (c(A\cup B)(x)\geq v\Rightarrow (cA)(x)\vee(cB)(x)\geq v)\\
\iff&\forall x\in X,v\in\sV:\; (c(A\cup B)(x)\geq v\Rightarrow\forall u\ll v:\, (cA)(x)\geq u \text{ or }(cB)(x)\geq u)\\
\iff&\forall v\in\sV:\;c^v(A\cup B)\subseteq \bigcap_{u\ll v}(c^uA\cup c^uB)
\end{align*}
the assertion of the Theorem follows from Proposition \ref{closuretowers}.
\end{proof}

\begin{rem}
(1) Of course, for $\sV$ ccd, (C0) follows from (C5) and is therefore not needed when characterizing $\sV$-approach spaces.

(2) For $v>\bot$, (C4) may be equivalently stated as
\begin{itemize}
\item[{\rm(C4)}] $c^v\emptyset = \bigcap_{u\ll v}\emptyset$,
\end{itemize}
and in this form
the requirement remains valid also when $v=\bot$: since there is no $u\ll \bot$ in $\sV$, trivially
$\bigcap_{u\ll\bot}\emptyset=X=c^{\bot}\emptyset$ (see Remark \ref{bottom}).
\end{rem}

When $\sV$ is completely distributive in the ordinary sense, then the conditions (C4), (C5) may be simplified, as follows. Recall that
an element $p\in \sV$ is {\em coprime} if
$$\forall D\subseteq \sV\text{  finite }:\;p\leq \bv D\Longrightarrow( \exists \,d\in D:p\leq d); $$
equivalently, if $p>\bot$, and $p\leq u\vee v$ always implies $p\leq u$ or $p\leq v$; or, equivalently, if
 $\{v\in\sV:v\not\geq p\}$ is a directed subset of $\sV$, that is: if any of its finite subsets has an upper bound in $\sV$. Note that, contrary to this definition, some authors regard also $\bot$ as coprime, but that does not affect the validity of the following well-known Proposition, for which one must grant the Axiom of Choice.

\begin{prop}{\rm (\cite{Gierz2003}, Theorem I-3.16.)}\label{coprime}
If $\sV$ is completely distributive, then $v=\bv\{p\in\sV\;|\;p\leq v, p \text{ coprime}\}$, for all $v\in\sV$.
\end{prop}

Now we can characterize $\sV$-approach spaces in terms closure towers satisfying (C1),(C2),(C3), and the following conditions {$(C4')$, $(C5')$}:
\begin{thm}\label{coprimeapproach}
Let $\sV$ be completely distributive. Then a $\sV$-closure space $(X,c)$ is a $\sV$-approach space
if, and only if, its closure tower $(c^v)_{v\in\sV}$ satisfies
\begin{enumerate}
\item[$\mathrm{(C4')}$] $c^p\emptyset=\emptyset$,
\item[$\mathrm{(C5')}$] $c^p(A\cup B)=c^pA\cup c^pB$,
\end{enumerate}
for all coprime elements $p\in\sV$ and $ A,B\subseteq X$.
\end{thm}
\begin{proof}
Firstly, \begin{align*}\forall x\in X:(c\emptyset)(x)=\bot
&\iff\forall x\in X, v\in\sV,v>\bot: v\nleq (c\emptyset)(x)\\
&\iff\forall x\in X, p\in\sV, p\text{ coprime}:p\nleq (c\emptyset)(x)\quad (\text{Proposition }\ref{coprime})\\
&\iff \forall p\in\sV, p\text{ coprime}:c^p(\emptyset)=\emptyset.
\end{align*}
Secondly, since trivially, for all coprime $p\in\sV, A, B \subseteq X, x\in X$,
\[x\in c^p(A\cup B)\iff p\leq c(A\cup B)(x) \]
and
\begin{align*}x\in (c^pA\cup c^pB)&\iff p\leq (cA)(x)\text{ or }p\leq (cB)(x)\\
&\iff p\leq (cA)(x)\vee (cB)(x).
\end{align*}
with Proposition \ref{coprime} one obtains that \[\forall p\in\sV, p\text{ coprime}:c^p(A\cup B)=c^pA\cup c^pB\iff \forall x\in X:c(A\cup B)(x)=(cA)(x)\vee (cB)(x).\]
\end{proof}

\begin{exmp}
(1) For the terminal quantale {\sf 1} one obtains ${\sf 1}\text{-}{\bf App}={\sf 1}\text{-}{\bf Cls}\cong{\bf Set}$.

(2) For the two-element chain ${\sf 2}$ (considered as a quantale with its frame structure, so that $\otimes=\wedge$), we see that ${\sf 2}\text{-}{\bf Cls}={\bf Cls}$ is the category of closure spaces, {\em i.e.,} of sets $X$ that come with an extensive, monotone and idempotent closure operation $c:\sP X\to \sP X$, and that
${\sf 2}\text{-}{\bf App}={\bf Top}$ is the category of topological spaces (presented in terms of a finitely additive closure operation), and their continuous maps.
\end{exmp}

For Lawvere's quantale $[0,\infty]$ we obtain from Lemma \ref{closureptw}, Proposition \ref{closuretowers}, Remark \ref{bottom}(2) and Theorem \ref{coprimeapproach} the following new characterizations of $[0,\infty]$-closure and approach spaces:

\begin{cor}\label{approach}
{\rm (1)} A $[0,\infty]$-closure space $X$ may be described in terms of a point-set-distance function $\delta$ satisfying
\begin{enumerate}
\item [{\rm (D1)}] $\delta(x,\{x\})=0,$
\item [{\rm (T')}] $\delta(x,\bigcup\CA)\leq {\rm sup}_{y\in B} {\rm inf}_{A\in \CA}(\delta(x,B)+\delta(y,A)).$
\end{enumerate}
for all $x\in X, B\subseteq X, \CA\subseteq \sP X$.
The $[0,\infty]$-closure space $X$ is a $[0,\infty]$-approach space if, and only if,  $\delta$ satisfies also
\begin{enumerate}
\item [{\rm (D2)}] $\delta(x,\emptyset)=\infty$,
\item [{\rm (D3)}] $\delta(x,A\cup B)={\rm min}\{\delta(x,A),\delta(x,B)\}$,
\end{enumerate}
for all $x\in X,\, A,B\subseteq X$; equivalently, if $X$ is an approach space in the ordinary sense, so that $\delta$ satisfies {\rm (D1)--(D4)}.

{\rm (2)} A $[0,\infty]$-closure space $X$ is equivalently described by a closure tower $(c^{\alpha}:\sP X\to\sP X)_{\alpha\in[0,\infty]}$ satisfying
\begin{enumerate}
\item[{\rm (C0)}]
$\text{if  }B\subseteq A, \text{ then  }c^{\alpha}B\subseteq c^{\alpha}A,$
\item[{\rm (C1)}] $\text{if  }{\rm inf}_{i\in I}{\beta}_i\leq\alpha, \text{ then   }\bigcap_{i\in I}c^{\beta_i}A\subseteq c^{\alpha}A,
$\item[{\rm (C2)}] $A\subseteq c^0A,$
\item [{\rm (C3)}] $c^{\alpha}c^{\beta}A\subseteq c^{\alpha+ \beta}A,$
\end{enumerate}
for all $A\subseteq X$ and $ \alpha,\beta, \beta_i \in [0,\infty]\; (i\in I)$. For $X$ to be an approach space, $(c^{\alpha})_{\alpha\in [0,infty]}$ must satisfy {\rm (C1)--(C3)} and
\begin{enumerate}
\item [{\rm (C4)}] $c^{\alpha}(\emptyset)=\emptyset$,
\item [{\rm (C5)}] $c^{\alpha}(A\cup B)=c^{\alpha}A\cup c^{\alpha}B$,
\end{enumerate}
for all $A,B\subseteq X, \alpha<\infty.$

{\rm (3)} A map $f:X\to Y$ of $[0,\infty]$-closure spaces $X,Y$, presented in terms of their respective closure towers
$(c^{\alpha}), (d^{\alpha})$, is contractive if, and only if, $f(c^{\alpha}A)\subseteq d^{\alpha} (f(A))$ for all $A\subseteq X, \alpha \in [0,\infty]$.
\end{cor}
In summary, $[0,\infty]\text{-}{\bf App}={\bf App}$ is the category of approach spaces (as defined in terms of point-set-distances) that may be equivalently described in terms of closure towers.

The quantale $[0,\infty]$ is of course isomorphic to the unit interval $[0,1]$, ordered by the natural $\leq$ and provided with the ordinary multiplication as $\otimes$.
Both, $[0,\infty]$ and $[0,1]$ are embeddable into the quantale $\bf {\Delta}$ of all {\em distance distribution functions} $\varphi: [0,\infty]\to[0,1]$, required to satisfy the left-continuity condition $\varphi(\beta)={\rm sup}_{\alpha<\beta}\varphi(\alpha)$, for all $\beta\in [0,\infty]$. Its order is inherited from $[0,1]$, and its monoid structure is given by the commutative
{\em convolution} product \[(\varphi\odot\psi)(\gamma)={\rm sup}_{\alpha+\beta\leq\gamma}\varphi(\alpha)\psi(\beta);\]
the $\odot$-neutral function $\kappa$ satisfies $\kappa(0)=0$ and $\kappa(\alpha)=1$ for all $\alpha >0$. We note $\kappa=\top$ in $\bf{\Delta}$ (so $\bf{\Delta}$ is integral), while the bottom element in ${\bf {\Delta}}$ has constant value $0$; we write $\bot=0$.
The significance of the quantale homomorphisms
$\sigma:[0,\infty]\to {\bf {\Delta}}$
and $\tau:[0,1]\to {\bf {\Delta}}$,
defined by $\sigma(\alpha)(\gamma)=0\text{ if }\gamma\leq\alpha\text{, and }1$ otherwise, and
$\tau(u)(\gamma)=u\text{ if }\gamma>0\text{, and }0$ otherwise, lies in the fact that every $\varphi\in {\bf{\Delta}}$ has a presentation 
$$\varphi=\bv_{\alpha\in[0,\infty]}\sigma(\alpha)\odot\tau(\varphi(\alpha))
=\bv_{\alpha\in(0,\infty)}\sigma(\alpha)\odot\tau(\varphi(\alpha)).$$ 
As a consequence (that was noted in \cite{LaxDistLaws}), one has a presentation of $\bf {\Delta}$ as a coproduct of $[0,\infty]$ and $[0,1]$ in the category {\bf Qnt} of quantales and their homomorphisms, with coproduct injections $\sigma$ and $\tau$, respectively.

The lattice $\bf {\Delta}$ is constructively completely distributive, hence completely distributive in the presence of the Axiom of Choice. The above presentation displays $\varphi$ as a join of coprime elements. Indeed, a distance distribution function $\pi$ is coprime if, and only if, there are $\alpha\in(0,\infty)$ and $u\in [0,1]$ such that $\pi=\sigma(\alpha)\odot\tau(u)$, {\em i.e.},
$\pi(\gamma)=\left\{
\begin{array}{ll}
0 & \text{if }\gamma\leq\alpha,\\
u & \text{if }\gamma> \alpha.
\end{array}
\right.$

A {\em probabilistic approach space} \cite{Jager, JagerApp} is a set $X$ equipped with a function $\delta: X\times\sP X\to {\bf {\Delta}}$, subject to
\begin{enumerate}[(PD1)]
\item $\forall x\in X: \;\delta(x,\{x\})=\kappa$,
\item $\forall x \in X: \;\delta(x,\emptyset)=0$,
\item $\forall x \in X,\; A, B\subseteq X:\;\delta(x,A\cup B)=\delta(x,A)\vee\delta(x,B)$,
\item $\forall x\in X, \; A \subseteq X,\; \varphi\in {\bf {\Delta}}:\;\delta(x,A)\geq \delta(x,A^{(\varphi)})\odot \varphi,
\text{ where } A^{(\varphi)}:=\{x\in X\;|\;\delta(x,A)\geq\varphi\}$.
\end{enumerate}
Calling a map $f:(X,\delta)\to (Y,\epsilon)$ of probabilistic approach spaces {\em contractive} when $\delta(x,A)\leq\epsilon(fx,f(A))$ for all $x\in X, A\subset X$, we obtain the category $\bf{ProbApp}$.

In analogy to Corollary \ref{approach}, the general results of this section lead to the following alternative descriptions of probabilistic approach spaces and their morphisms.

\begin{cor}\label{probapp}
{\rm (1)} A function $\delta:X\times \sP X\to{\bf{\Delta}}$ is a probabilistic approach structure on a set $X$ if, and only if, $\delta$ satisfies {\rm (PD1),(PD2),(PD3)} and
\begin{enumerate}
\item[{\rm (T')}] $\forall \CA\subseteq \sP X,\, B\subseteq X:\,\delta(x,\bigcup\CA)\geq \bw_{y\in B} \bv_{A\in \CA}(\delta(x,B)\odot\delta(y,A)).$
\end{enumerate}
Equivalently, the function $c:\sP X\to{\bf{\Delta}}^X$ with $(cA)(x)=\delta(x,A)$ makes $(X,c)$ a ${\bf{\Delta}}$-approach space.

{\rm (2)} The probabilistic approach structure on a set $X$ may be described equivalently by a family of functions $c^{\varphi}:\sP X\to \sP X\; (\varphi\in {\bf{\Delta}})$ satisfying
\begin{enumerate}
\item[{\rm (PC1)}] $\text{if  }\varphi\leq\bv_{i\in I}{\psi}_i, \text{ then   }\bigcap_{i\in I}c^{\psi_i}A\subseteq c^{\varphi}A,
$\item[{\rm (PC2)}] $A\subseteq c^{\kappa}A,$
\item [{\rm (PC3)}] $c^{\varphi}c^{\psi}A\subseteq c^{\varphi\odot \psi}A,$
\item [{\rm (PC4)}] $c^{\pi}(\emptyset)=\emptyset$,
\item [{\rm (PC5)}] $c^{\pi}(A\cup B)=c^{\pi}A\cup c^{\pi}B$,
\end{enumerate}
for all $A,B\subseteq X,\, \varphi,\psi,\pi\in{\bf{\Delta}},\, \pi$ coprime.

{\rm (3)} A map $f:X\to Y$ of probabilistic spaces $X,Y$, presented in terms of their respective closure towers
$(c^{\varphi}), (d^{\varphi})$, is contractive if, and only if, $f(c^{\varphi}A)\subseteq d^{\varphi} (f(A))$ for all $A\subseteq X, \varphi \in {\bf{\Delta}}$.
\end{cor}
In summary, ${\bf ProbApp}={\bf{\Delta}}\text{-}{\bf App}$, and contractivity of a map is equivalently described by continuity with respect to ${\bf{\Delta}}$-closure towers.

\section{$\sV$-approach spaces via ultrafilter convergence}
{\em Throughout this section, the quantale $\sV$ is assumed to be completely 
distributive.} 

We let $\bbU=(\sU,\Sigma, \dot{(\text{-})})$ denote the ultrafilter monad on $\bf{Set}$. Hence, $\sU X$ is the set of ultrafilters on the set $X$, and the effect of $\sU$ on a map $f:X\to Y$ and the monad structure of $\sU$ are described by
\begin{align*}
&\sU f:\sU X\to\sU Y,\quad\frak{x}\mapsto f[\frx],\quad(B\in f[\frx]\iff f^{-1}B\in\frx),\\
&\dot{(\text{-})}: X\to\sU X, \quad x\mapsto\dot{x},\quad (A\in\dot{x}\iff x\in A),\\
&\Sigma_X: \sU\sU X\to\sU X,\quad \frak{X}\mapsto \Sigma\frak{X},
\quad(A\in\Sigma\frak{X}\iff\{\frak{x}\in\sU X\;|\;A\in\frak{x}\}\in\frak{X}),
\end{align*}
for all $x\in X$, $\frx\in\sU X$, $\frak{X}\in\sU\sU X$, $A\subseteq X, B\subseteq Y$.

\begin{prop}\label{Ufilterdistributes}
The ultrafilter monad $\bbU$ distributes laxly over the $\sV$-powerset monad $\bbP_{\sV}$, via
\[\beta_X:\sU(\sV^X)\to\sV^{\sU X}, \quad (\beta_X\frak{s})(\frak{x})=\bigwedge_{S\in\frak{s}\atop A\in\frak{x}}\bigvee_{\sigma\in S\atop x\in A}\sigma(x)\quad(\frak{s}\in\sU(\sV^X),\; \frak{x}\in\sU X).\]
\end{prop}

\begin{proof} We verify the defining conditions (a)--(f) of Section 2.

(a) 
With $f[\frak{x}]=\frak{y}$ one has
\[(\beta_Y\cdot(\sU f_!)(\frak{s}))(\frak{y})
=\bw_{T\in f_![\frak{s}]\atop B\in\frak{y}}\bv_{\tau\in T\atop y\in B}\tau(y)
=\bw_{S\in\frak{s}\atop A\in\frak{x}}\bv_{\sigma\in S\atop x\in A}f_!(\sigma)(fx)
\geq \bw_{S\in\frak{s}\atop A\in\frak{x}}\bv_{\sigma\in S\atop x\in A}\sigma(x).
\]
Consequently,
\[(\beta_Y\cdot(\sU f_!)(\frak{s}))(\frak{y})\geq\bv_{f[\frak{x}]=\frak{y}}\bw_{S\in\frak{s}\atop A\in\frak{x}}\bv_{\sigma\in S\atop x\in A}\sigma(x)=\bv_{ f[\frak{x}]=\frak{y}}\beta_X(\frak{s})(\frak{x})=((\sU f)_!\cdot\beta_X(\frak{s}))(\frak{y}).\]

(b) Since, for $\frx,\fry\in \sU X$, one has $\sy_{\sU X}(\frak{x})(\frak{y})=\sk$ if $\frak{y}=\frak{x}$, and $\bot$ otherwise,
$\sy_{\sU X}\leq\beta_X\cdot\sU\sy_X$ follows from 
\[(\beta_X\cdot\sU\sy_X(\frak{x}))(\frak{x})=\bw_{S\in\sy_X[\frak{x}]\atop A\in \frak{x}}\bv_{\sigma\in S\atop x\in A}\sigma(x)=\bw_{A,B\in\frak{x}}\bv_{x\in A\atop y\in B}\sy_X(y)(x)\geq\sk,\]
with the last inequality following from $A\cap B\neq\emptyset$ for all $A, B\in\frak{x}$.

(c) For all $\frak{w}\in \sU(\sV^{\sV^X})$ and $\frak{x}\in\sU X$,
\begin{align*}
(\sfs_{\sU X}\cdot(\beta_X)_!\cdot\beta_{\sV^X}(\frak{w}))(\frak{x})
&=\bv_{\rho\in\sV^{\sU X}}((\beta_X)_!\cdot\beta_{\sV^X}(\frak{w}))(\rho)\otimes\rho(\frak{x})\\
&=\bv_{\rho\in\sV^{\sU X}}\bv_{\frak{s}\in\sU(\sV^X)\atop\beta_X(\frak{s})=\rho}\beta_{\sV^X}(\frak{w})(\frak{s})\otimes\rho(\frak{x})\\
&=\bv_{\frak{s}\in\sU(\sV^X)}\beta_{\sV^X}(\frak{w})(\frak{s})\otimes\beta_X(\frak{s})(\frak{x})\\
&=\bv_{\frak{s}\in\sU(\sV^X)}
\Big(\bw_{\bbS\in\frak{w}\atop S\in\frak{s}}\bv_{\Phi\in\bbS\atop \sigma\in S}\Phi(\sigma)\Big)\otimes\Big(\bw_{S\in\frak{s}\atop A\in\frak{x}}\bv_{\sigma\in S\atop x\in A}\sigma(x)\Big),
\end{align*}
while
\[(\beta_X\cdot\sU\sfs_X(\frak{w}))(\frak{x})=\bw_{S\in\sfs_X[\frak{w}]\atop A\in\frx}\bv_{\sigma\in S\atop x\in A}\sigma(x)
=\bw_{\bbS\in\frak{w}\atop A\in \frx}\bv_{\Phi\in\bbS\atop x \in A}\sfs_X(\Phi)(x)
=\bw_{\bbS\in\frak{w}\atop A\in \frx}\bv_{\Phi\in\bbS\atop x \in A}\bv_{\sigma\in\sV^X}\Phi(\sigma)\otimes\sigma(x).
\]
Consequently, in order for us to conclude $ (\sfs_{\sU X}\cdot(\beta_X)_!\cdot\beta_{\sV^X}(\frak{w}))(\frak{x})\leq
(\beta_X\cdot\sU\sfs_X(\frak{w}))(\frak{x})$, it suffices to show that, given any $\frak{s}\in\sU(\sV^X), \bbS\in\frak{w}, A\in \frx$, as well as $u\ll\beta_{\sV^X}(\frak{w}),v\ll\beta_X(\frak{s})(\frx)$ in $\sV$, that there are $\Phi\in\bbS,\sigma\in\sV^X,x\in A$ with $u\otimes v\leq\Phi(\sigma)\otimes\sigma(x)$.
But indeed, from the stated hypothesis on $u,v\in\sV$, for all $S\in\frak{s}$ one obtains $\Phi_S\in\bbS,\sigma_S\in S$ with $u\leq\Phi_S(\sigma_S)$, and $\tau_S\in S, x_S\in A$ with $v\leq\tau_S(x_S)$.  
Now, the set $M=\{\sigma_S\;|\;S\in\frak{s}\}$ satisfies $M\cap S\neq\emptyset$ for all $S\in\frak{s}$ and must therefore belong to $\frak{s}$ 
(since, otherwise, we could find an ultrafilter properly containing $\frak{s}$); likewise, $N=\{\tau_S\;|\;S\in\frak{s}\}\in\frak{s}$. Consequently, $M\cap N\neq\emptyset$, from which one derives the needed claim.

(d) Let $\sigma\in\sV^X, \frx\in\sU X$. If $\frx=\dot{x}$ for $x\in X$, then
\[\dot{(\text{-})}_!(\sigma)(\frx)=\bv_{y\in X\atop\dot{y}=\frx}\sigma(y)=\sigma(x)=\bw_{S\in\dot{\sigma}\atop A\in\frx}\bv_{\tau\in S\atop y\in A}\tau(y)
=\beta_X(\dot{\sigma})(\frx);\]
otherwise $\dot{(\text{-})}_!(\sigma)(\frx)=\bot$, and the needed inequality holds trivially.

(e) If for $\frak{X}\in\sU\sU X,\frak{S}\in\sU\sU(\sV^X)$ we have $\Sigma_X(\frak{X})=\frx$, $\Sigma_{\sV^X}(\frak{S})=\frak{s}$, then for any given $S\in\frak{s}, A\in\frak{x}$, there are $\CS_0\in\frak{S}$, $\CA_0\in\frak{X}$ such that $S\in\frak{t},A\in\frak{y}$ for all $\frak{t}\in\CS_0, \frak{y}\in\CA_0$. Obviously then,
 \[\bv_{\frak{t}\in\CS_0\atop \frak{y}\in\CA_0}\bw_{T\in\frak{t}\atop B\in\frak{y}}\bv_{\tau\in T\atop y\in B}\tau(y)\leq\bv_{\sigma\in S\atop x\in A}\sigma(x).\]
Consequently, for all $\frak{S}\in\sU\sU(\sV^X),\;\frx\in\sU X$, putting $\frak{s}=\Sigma_{\sV^X}(\frak{S})$ one obtains  
\begin{align*}
((\Sigma_X)_!\cdot \beta_{\sU X}\cdot\sU\beta_X(\frak{S}))(\frak{x})
&=\bv_{\frak{X}\in\sU\sU X,\;\Sigma_X(\frak{X})=\frak{x}}(\beta_{\sU X}\cdot\sU\beta_X(\frak{S}))(\frak{X})\\
&=\bv_{\Sigma_X(\frak{X})=\frak{x}}\bw_{F\in\beta_X[\frak{S}]\atop\CA\in\frak{X}}\bv_{\varphi\in F\atop \frak{y}\in\CA}\varphi(\frak{y})\\
&=\bv_{\Sigma_X(\frak{X})=\frak{x}}\bw_{\CS\in\frak{S}\atop \CA\in\frak{X}}\bv_{\frak{t}\in\CS\atop \frak{y}\in\CA}\beta_X(\frak{t})(\frak{y})\\
&=\bv_{\Sigma_X(\frak{X})=\frak{x}}\bw_{\CS\in\frak{S}\atop \CA\in\frak{X}}\bv_{\frak{s}\in\CS\atop \frak{y}\in\CA}\bw_{T\in\frak{s}\atop B\in\frak{y}}\bv_{\tau\in T\atop y\in B}\tau(y)\\
&\leq\bw_{S\in\frak{s}\atop A\in\frx}\bv_{\sigma\in S\atop x\in A}\sigma(x)\\
&=\beta_X\cdot\Sigma_{\sV^X}(\frak{S})(\frak{a}).\\
\end{align*}

(f) For $g,h: Z\to\sV^X$ with $g\leq h$ and all $\frak{z}\in\sU Z,\frak{x}\in\sU X$, one has
\[(\beta_X\cdot\sU g(\frak{z}))(\frx)=\bw_{C\in\frak{z}\atop A\in\frx}\bv_{z\in C\atop x\in A}(gz)(x)\leq\bw_{C\in\frak{z}\atop A\in\frx}\bv_{z\in C\atop x\in A}(hz)(x)=(\beta_X\cdot\sU h(\frak{z}))(\frx).\]
\end{proof}

\begin{rem}\label{Barr extension}
The lax extension $\overline{\sU}:\sV\text{-}{\bf Rel}\to\sV\text{-}{\bf Rel}$ of $\bbU$ corresponding to $\beta$ 
(see Remark \ref{lax ext}) is given by
\[\overline{\sU}r(\frx,\fry)=\bw_{A\in\frx,B\in\fry}\bv_{x\in A,y\in B}r(x,y),\]
for all $r:X\nrightarrow Y,\;\frx\in\sU X,\fry\in\sU Y.$
\end{rem}

A straightforward calculation gives a description of lax $(\beta,\sV)$-algebras (with $\beta$ as in Proposition \ref{Ufilterdistributes}), in analogy to the description of $(\alpha,\sV)$-algebras of Lemma \ref{closureptw}:.

\begin{lem}\label{beta algebras}
A map $\ell:\sU X\to\sV^X$ makes $(X,\ell)$ a lax $(\beta,\sV)$-algebra if and, only if, $\ell$ satisfies

\begin{enumerate}
\item[{\rm (R'')}] $\forall x\in X:\quad\quad\quad\quad\quad\quad\quad\quad\sk\leq \ell(\dot{x})(x)$,
\item[{\rm (T'')}] $\forall \frak{X}\in\sU\sU X,\frak{y}\in\sU X,z\in X:\Big(\bw_{\mathcal{A}\in\frak{X}, B\in\frak{y}}\limits\bv_{\frak{x}\in\mathcal{A}, y\in B}\limits(\ell\frak{x})(y)\Big)\otimes(\ell \frak{y})(z)\leq\ell (\Sigma_X\frak{X})(z)$.
\end{enumerate}
A map $f:X\to Y$ is a lax homomorphism $f:(X,\ell)\to(Y,\ell')$ of lax $(\beta,\sV)$-algebras if, and only if,
\begin{enumerate}
\item[\rm{(M'')}] $\forall \frak{x}\in\sU X,y\in X:\quad\quad\quad\quad\;(\ell\frak{x})(y)\leq(\ell' f[\frak{x}])(fy)$.
\end{enumerate}

\end{lem} 




When displayed in terms of the $\sV$-relation $a:\sU X\nrightarrow X$ with $a(\frx,y)=(\ell\frx)(y)$, (see Remark \ref{TVCat}), conditions (R"), (T") read as
\[\sk\leq a(\dot{x},x)\quad\text{    and   }\quad\overline{\sU}a(\frak{X},\fry)\otimes a(\fry, z)\leq a(\Sigma\frak{X},z),\quad\quad\quad(*)\]
for all $x,z\in X, \fry \in \sU X, \frak{X}\in\sU\sU X$.

Next we will establish an adjunction between the categories $(\alpha,\sV)\text{-}{\bf Alg}\cong(\bbP,\sV,\hat{\sP})\text{-}{\bf Cat}\cong \sV\text{-}{\bf Cls}$ and 
$(\beta,\sV)\text{-}{\bf Alg}\cong (\bbU,\sV,\overline{\sU})\text{-}{\bf Cat}$, the restriction of which will then give an isomorphism $$\sV\text{-}{\bf App}\cong (\beta,\sV)\text{-}{\bf Alg}.$$

First recall from \cite{LaxDistLaws} that, given lax extensions $\hat{S},\hat{T}$ of $\bf Set$-monads $\bbS, \bbT$ to
$\sV\text{-}{\bf Rel}$,
an {\em algebraic morphism} $h:(\bbS,\hat{S})\to({\mathbb T},\hat{T})$ is a family of $\sV$-relations $h_X:SX\nrightarrow TX\;(X\in {\bf Set})$,
satisfying the following conditions for all $f:X\to Y\text{ in } {\bf Set}\text{   and  }r:X\nrightarrow Y,\;a:TX\nrightarrow X \text{ in }\sV\text{-}{\bf Rel}$:

 \begin{tabular}{llr}
 &&\\
 a. & $Tf\circ h_X\leq  h_Y\circ Sf$, &\quad\quad\quad\quad
 (lax naturality)\\
 b. & $e_X\leq h_X\circ d_X$,  &\quad (lax unit law)\\
 c. & $m_X\circ h_{TX}\circ\hat{S}h_X\leq h_X\circ n_X$, &\quad\quad\quad\quad\quad\quad\quad\; (lax multiplication law)\\
 d. & $h_Y\circ\hat{S}r\leq \hat{T}r\circ h_X$, &\quad (lax compatability)\\
 e. & $\hat{S}(a\circ h_X)\leq\hat{S}a\circ\hat{S}h_X$. &\quad (strictness of $\hat{S}$ at $h$)\\
 &&\\
 \end{tabular}
 
 Here, for $s:Y\nrightarrow Z$, the composite $s\circ r:X\nrightarrow Z$ in $\sV\text{-}{\bf Rel}$ is given by $(s\circ r)(x,z)=\bv_{y\in Y}s(y,z)\otimes r(x,y)$, and (as in \cite{MonTop}) we identify a map $f:X\to Y$ with its $\sV$-graph $f_{\circ}:X\nrightarrow Y$, given by $f_{\circ}(x,y)=\sk$ if $fx=y$, and $\bot$ otherwise. Now, such lax transformation $h:\hat{S}\to\hat{T}$ induces the {\em algebraic functor}
 \[A_h:(\bbT,\sV,\hat{T})\text{-}{\bf Cat}\to(\bbS,\sV,\hat{S})\text{-}{\bf Cat},\; (X,a)\mapsto (X,a\circ h_X).\]
Considering ${\mathbb S}={\mathbb P},{\mathbb T}={\mathbb U}$, let us consider $\varepsilon_X:\sP X\nrightarrow\sU X$
by
\[\varepsilon_X(A,\frx)=\left\{
\begin{array}{ll}
\sk & \text{if }A\in\frx\\
\bot & \text{otherwise}
\end{array}
\right\},\]
for all $A\subseteq X, \frx\in \sU X$.
\begin{prop}
$\varepsilon:(\bbP,\hat{\sP})\to(\bbU,\overline{\sU})$ is an algebraic morphism and, hence, induces the algebraic functor
\[A_{\varepsilon}:(\bbU,\sV,\overline{\sU})\text{-}{\bf Cat}\to(\bbP,\sV,\hat{\sP})\text{-}{\bf Cat},\quad(X,a)\mapsto (X,a\circ\varepsilon_X).\]
\end{prop}

\begin{proof}
We verify conditions a--e above.

a. Trivially, if $A\in\frx\in\sU X$ and $f[\frx]=\fry$, then $f(A)\in\fry$, and $(\sU f\circ\varepsilon_X)(A,\fry)\leq (\varepsilon_Y\circ\sP f)(A,\fry)$ follows. 

b. Likewise, if $\frx=\dot{x}$, then $\{x\}\in\frx$, and $(\dot{\text{-}})_X(x,\frx)\leq\varepsilon_X\circ\{\text{-}\}_X(x,\frx)$ follows.

c. For $\frak{X}\in\sU\sU X, \CA\subseteq\sP X$ one has
\[(\varepsilon_{\sU X}\circ\hat{\sP}\varepsilon_X)(\CA,\frak{X})
=\bv_{\CB\subseteq \sU X}\hat{\sP}\varepsilon_X(\CA,\CB)\otimes\varepsilon_{\sU X}(\CB,\frak{X})
=\bv_{\CB\in\frak{X}}\hat{\sP}\varepsilon_X(\CA,\CB)=\bv_{\CB\in\frak{X}}\bw_{\fry \in\CB}\bv_{A\in\CA}\varepsilon(A,\fry)=\sk\]
in the case that, for all $\fry\in\CB$, there is $A\in\CA$ with $A\in\fry$, and $\bot$ otherwise. So, in the former case, given any $\CB\in\frak{X}$, one has $\CB\subseteq\{\fry\in\sU X\;|\;\bigcup\CA\in\fry\}\in\frak{X}$ and, hence, $\bigcup\CA\in\frx:=\Sigma\frak{X}.$ Consequently,
$(\Sigma_X\circ\varepsilon_{\sU X}\circ\hat{\sP}\varepsilon_X)(\CA,\frx)\leq(\varepsilon_X\circ\bigcup_X)(\CA,\frx).$

d. For $r:X\nrightarrow Y, A\subseteq X,\fry\in\sU Y$, we must compare
\[(\varepsilon_Y\circ\hat{\sP}r)(A,\fry)=\bv_{B\subseteq Y}\varepsilon_Y(B,\fry)\otimes\hat{P}r(A,B)
=\bv_{B\in\fry}\bw_{y\in B}\bv_{x\in A}r(x,y) \]
with
\[(\overline{\sU}r\circ\varepsilon_X)(A,\fry)=\bv_{\frx\in \sU X}\overline{\sU}r(\frx,\fry)\otimes\varepsilon_X(A,\frx)               
=\bv_{\frx\in\sU X\atop\frx\ni A}\bw_{A'\in\frx\atop B'\in\fry}\bv_{x'\in A'\atop y'\in B'}r(x',y').    \]
So, given $B\in\fry$, we consider $u\ll \bw_{y\in B}\bv_{x\in A}r(x,y)$ in $\sV$.
For all $y\in B$ we may then pick $fy\in A$ with $u\leq r(fy,y)$. With the map $f:B\to A$ we choose an ultrafilter $\frx$
on $X$ that contains all the sets $f(C),\;C\in\fry,C\subseteq B.$ For all such $C$ and any $B'\in\fry$, 
since $C\cap B'\neq\emptyset$, we finally obtain some $y'\in B'$ with $u\leq r(fy',y')$. Now $(\varepsilon_Y\circ\hat{\sP}r)(A,\fry)
\leq (\overline{\sU}r\circ\varepsilon_X)(A,\fry)$ follows.

e. For $\CA\subseteq\sP X, B\subseteq X$ one has
\[\hat{\sP}(a\circ\varepsilon_X)(\CA,B)
=\bw_{y\in B}\bv_{A\in \CA}\bv_{\frx\in\sU X}a(\frx,y)\otimes\varepsilon_X(A,\frx)
=\bw_{y\in B}\bv_{A\in \CA}\bv_{\frx\ni A}a(\frx,y),\]
while
\[(\hat{\sP}a\circ\hat{\sP}\varepsilon_X)(\CA,B)=\bv_{\CB\subseteq \sU X}\hat{\sP}a(\CB,B)\otimes\hat{\sP}\varepsilon_X(\CA,\CB)
=\bv_{\CB\subseteq\sU X}\Big(\bw_{y\in B}\bv_{\fry\in\CB}a(\fry,y)\Big)\otimes\Big(\bw_{\fry'\in\CB}\bv_{a\in\CA}\varepsilon_X(A,\fry')\Big).\]
Now, whenever $u\ll \hat{\sP}(a\circ\varepsilon_X)(\CA,B)$ in $\sV$, for all $y\in B$ one obtains 
$\frx_y\in  A_y\in\CA$ with $u\leq a(\frx_y,y)$. Putting $\CB:=\{\frx_y\;|\;y\in B\}$ one sees $u\leq (\hat{\sP}a\circ\hat{\sP}\varepsilon_X)(\CA,B)$, which gives the needed inequality.
\end{proof}

\begin{rem}\label{rightadjoint}
In the quantaloid (see \cite{Stubbe2005}) $\sV\text{-}{\bf Rel}$,  the sup-map 
\[\sV\text{-}{\bf Rel}(\sU X, X)\to\sV\text{-}{\bf Rel}(\sP X, X),\; a\mapsto a\circ\varepsilon_X,\]
has a right adjoint, which assigns to $\delta:\sP X\nrightarrow X$ the $\sV$-relation $\delta\swa\varepsilon_X:\sU X\nrightarrow X$, given by
\[(\delta\swa\varepsilon_X)(\frx,x)=\bw_{A\subseteq X}(\delta(A,x)\swa\varepsilon_X(A,\frx))=\bw_{A\in\frx}\delta(A,x).\]
Writing $(cA)(x)$ for $\delta(A,x)$ we will take advantage of this obvious fact in the proof of the Theorem below.
\end{rem}

If we describe $(\bbP,\sP)$-categories as $\sV$-closure spaces and $(\bbU,\sV)$-categories as lax $(\beta,\sV)$-algebras then, $A_{\varepsilon}$ takes the form
\[A_{\varepsilon}: (\beta,\sV)\text{-}{\bf Alg}\to\sV\text{-}{\bf Cls},\quad(X,\ell)\mapsto(X,c_{\ell}:\sP X\to\sV^X), \;(c_{\ell}A)(x)=\bv_{\frx\in\sU X, \frx\ni A}(\ell\frx)(x).\]
We are now ready to prove the main result of this paper:

\begin{thm}\label{mainthm}
For a completely distributive quantale $\sV$, the algebraic functor $A_{\varepsilon}$ embeds $(\beta,\sV)\text{-}{\bf Alg}$ into $\sV\text{-}{\bf Cls}$ as a full coreflective subcategory,
 which is precisely the category $\sV\text{-}{\bf App}$ of $\sV$-approach spaces.
\end{thm}
 
\begin{proof}
That $A_{\varepsilon}$ actually takes values in $\sV\text{-}{\bf App}$ is just a reflection of the fact that, for any ultrafilter $\frx$ on a set $X$, one has $\emptyset\notin \frx$, and $A\cup B\in\frx$ only if $A\in\frx$ or $B\in\frx$.
Next we prove that $A_{\varepsilon}$ has a right adjoint, described by (see Remark \ref{rightadjoint})
\[R:\sV\text{-}{\bf Cls}\to (\beta,\sV)\text{-}{\bf Alg},\quad (X,c)\mapsto (X,\ell_c:\sU X\to\sV^X),
\;(\ell_c\frx)(x)=\bw_{A\in\frx}(cA)(x).\]
Given $(X,c)\in\sV\text{-}{\bf Cls}$ we must first show $(X,\ell_c)\in (\beta,\sV)\text{-}{\bf Alg}$, that is: writing $a(\frx,y)$ for $(\ell_c\frx)(y)$, we must establish $(*)$ of Lemma \ref{beta algebras}. Trivially, 
$\sk\leq\bv_{A\in\dot{x}}(cA)(x)=a(\dot{x},x)$ for all $x\in X$. Our strategy to show 
\[\overline{\sU}a(\frak{X},\fry)\otimes a(\fry,z)\leq a(\Sigma\frak{X},z)=\bw_{A\in\Sigma\frak{X}}(cA)(z)\] 
for all $\frak{X}\in\sU\sU X, \fry\in\sU X, z\in X,$ is to consider any $A\in\Sigma\frak{X}$ and 
\[u\ll\overline{\sU}a(\frak{X},\fry)=
\bw_{\CA\in\frak{X}\atop B\in\fry}\bv_{\frx\in\CA\atop y\in B}a(\frx,y)\]
in $\sV$ and to show
$u\otimes a(\fry,z)\leq (cA)(z).$ Indeed, with $A\in\Sigma\frak{X}$ one has $\CA:=\{\frx\in\sU X\;|\;A\in\frx\}\in\frak{X}$. 
Then, putting $C:=\{y\in X\;|\;\exists \frx\in\CA :u\leq a(\frx,y)\}$, since $a(\frx,y)\leq(cA)(y)$ whenever $A\in\frx$, we obtain $C\subseteq c^uA=\{y\in X\;|\;(cA)(y)\geq u\}$. Since $u\ll \bw_{B\in\fry}\bv_{\frx\in\CA, y\in B}a(\frx,y)$,
so that for all $B\in\fry$ there is $\frx\in\CA$ with $u\leq a(\frx,y)$, we see that $B\cap C\neq\emptyset$ whenever $B\in
\fry$. Maximality of $\fry$ therefore forces $C\in\fry$, and then $c^uA\in\fry$. With (C3') of Remark \ref{bottom} we conclude
\[u\otimes a(\fry,z)=u\otimes\bw_{B\in\fry}(cB)(z)\leq u\otimes \,c(c^uA)(z)\leq(cA)(z).\]
For the adjunction $A_{\varepsilon}\dashv R$, it now suffices to show that, given a $\sV$-closure space $(X,c)$ and a $(\beta,\sV)$-algebra $(Y,\ell)$, a map $f:Y\to X$ is a morphism $A_{\varepsilon}(Y,\ell)\to (X,c)$ in 
$\sV\text{-}{\bf Cls}$ if, and only if, it is a morphism $(Y,\ell)\to R(X,c)$ in $(\beta,\sV)\text{-}{\bf Alg}$.
This, however, is obvious, since either statement means equivalently
\[\forall B\in\fry\in\sU Y,\,y\in Y: (\ell\fry)(y)\leq c(f(B))(fy).\]

Next, for a $\sV$-approach space $(X,c)$, we must show $A_{\varepsilon}R(X,c)=(X,c)$, that is: $c_{\ell_c}=c$. Since the adjunction gives $c_{\ell_c}\leq c$, it suffices to show $``\geq"$, that is:
for all $A\subseteq X, x\in X$, 
\[(cA)(x)\leq\bv_{\frx\in\sU X\atop\frx\ni A}\bw_{B\in\frx}(cB)(x),\] 
and for that, by Proposition \ref{coprime}, it suffices to check that every coprime element $p$ in $\sV$ with $p\leq (cA)(x)$ satisfies
$p\leq\bv_{\frx\ni A}\bw_{B\in\frx}(cB)(x).$ 
But the set $\mathcal{I}_p=\{B\subseteq X:\delta(B,x)\not\geq p\}\subseteq PX$ is directed since $p$ is coprime, and $\mathcal{I}_p$ is disjoint from the filter $\{B\subseteq X:A\subseteq B\}$.  
There is therefore an ultrafilter $\frx_p$ with $A\in\frx_p$ disjoint from $\mathcal{I}_p$. Thus, for all $B\in\frx_p$, $(cB)(x)\geq p$ and, consequently, 
\[p\leq\bw_{B\in\frx_p}(cB)(x)\leq\bv_{\frx\ni A}\bw_{B\in\frx}(cB)(x).\]

Finally we show $RA_{\varepsilon}(X,\ell)=(X,\ell)$ for every $(X,\ell)\in(\beta,\sV)\text{-}{\bf Alg}$, that is: 
$\ell_{c_{\ell}}=\ell$. As the adjunction gives $``\geq"$, we need to show only $\ell_{c_{\ell}}\leq\ell$.
Writing $a(\frx,y)$ for $(\ell\frx)(y)$, this means that, for all $\frx\in\sU X, x\in X$, we must prove
\[(\ell_{c_\ell}\frx)(x)=\bw_{A\in\frx}\bv_{\fry\in\sU X\atop\fry\ni A}a(\fry,x)\leq a(\frx,x).\]
To this end, considering any $u\ll (\ell_{c_\ell}\frx)(x)$ in $\sV$, for all $A\in\frx$ one obtains $\fry_A\in\sU X$ with $A\in\fry_A$ and $u\leq a(\fry_A,x)$. So, for all $A\in\frx$, the sets
\[\CA_A=\{\fry\in\sU X\;|\;A\in\fry, u\leq a(\fry,x\}\]
are not empty, and we can choose an ultrafilter $\frak{X}$ on $\sU X$ containing all of them. Since for every 
$A\in \frx$ one has $\{\fry\in\sU X\;|\;A\in \fry\}\supseteq \CA_A\in\frak{X}$, we obtain
$\Sigma\frak{X}=\frx$. Furthermore,
\[\overline{U} a(\frX,\dot{x})=\bw_{\mathcal{A}\in\frX\atop B\in\dot{x}}\bv_{\fry\in\mathcal{A}\atop y\in B}a(\fry,y)=\bw_{\mathcal{A}\in\frX}\bv_{\fry\in\mathcal{A}}a(\fry,x)
=\bw_{A\in\frx}\bv_{\fry\in\mathcal{A}_A}a(y,x)\geq u.\]
With the transitivity of $a$ we conclude
\[a(\frx,x)=a(\Sigma\frX,x)\geq \overline{\sU} a(\frX,\dot{x})\otimes a(\dot{x},x)\geq u\otimes\sk=u,\]
and $a(\frx,x)\geq(\ell_{c_{\ell}}\frx)(x)$ follows, as desired.
\end{proof}

The isomorphism
\[\sV\text{-}{\bf App}\cong(\bbU,\sV,\overline{\sU})\text{-}{\bf Cat}\cong(\beta,\sV)\text{-}{\bf Alg}\]
gives Barr's \cite{Barr} description of topological spaces and the Clementino-Hofmann \cite {CleHof2003} presentation of approach spaces in terms of ultrafilter convergence when one chooses $\sV={\sf 2}$ and $\sV=[0,\infty]$, respectively. For $\sV={\bf {\Delta}}$ we obtain the corresponding description of {\bf ProbApp}, as follows.
\begin{cor}\label{ProbAppUltra}
The structure of a probabilistic approach space on a set $X$ may be described equivalently as a map 
$\ell:\sU X\to {\bf{\Delta}}^X$ satisfying, for all $\frX\in\sU\sU X, \fry\in\sU X, z\in X$,
\begin{enumerate}
\item[{\rm (R")}] $\kappa\leq(\ell\frx)(x),$
\item[{\rm (T")}] $\Big(\bw_{\CA\in\frX\atop B\in\fry}\bv_{\frx\in\CA\atop\y\in B}(\ell\frx)(y)\Big)\odot(\ell\fry)(z)\leq\ell(\Sigma\frX)(z).$
\end{enumerate}
A map $f:X\to Y$ of probabilistic approach spaces $(X,\ell), (Y,\ell')$ is contractive precisely when, for all $\frx\in\sU X, x\in X$,
\begin{enumerate}
\item[{\rm (M")}] $(\ell\frx)(x)\leq(\ell'f[\frx])(fx).$
\end{enumerate}
\end{cor}

\section{Change-of-base functors}

For a monad $\bbT=(T,m,e)$ and a quantale $\sV$, let us call a set $X$ equipped with a map $c:TX\to\sV^X$ a 
$(\bbT,\sV)${\em -graph}. With a morphism $f:(X,c)\to(Y,d)$ required to satisfy (M) of Definition \ref{defnlaxalgebra}, we obtain the category $(\bbT,\sV)\text{-}{\bf Gph}$, 
which contains $(\lambda,\sV)\text{-}{\bf Alg}$ as a full subcategory, for any 
lax distributive law $\lambda$ of $\bbT$ over $\bbP_{\sV}$. For a monotone map $\varphi:\sV\to\sW$ one has the 
{\em change-of-base functor}
\[B_{\varphi}:(\bbT,\sV)\text{-}{\bf Gph}\to(\bbT,\sW)\text{-}{\bf Gph}, \;(X,c)\mapsto (X,\varphi^X\cdot c),\]
with $\varphi^X:\sV^X\to\sW^X,\;\sigma\mapsto \varphi\cdot\sigma.$ Since $(\bbT,\sV)$-graphs actually neither refer to the monad structure of the functor $T$ nor to the quantalic structure of the lattice 
$\sV$, they behave well under any adjunction of monotone maps:

\begin{lem}\label{graphadjunction}
If $\varphi\dashv\psi:\sW\to\sV$, then $B_{\varphi}\dashv B_{\psi}:(\bbT,\sW)\text{-}{\bf Gph}\to(\bbT,\sV)\text{-}{\bf Gph}$.
\end{lem}

\begin{proof}
Given a $(\bbT,\sV)$-graph $(X,c)$ and a $(\bbT,\sW)$-graph, we must verify that a map $f:X\to Y$ is a morphism
 $(X,c)\to(Y,\psi^Y\cdot d)$ if, and only if, it is a morphism $(X,\varphi^X\cdot c)\to(Y,d)$, which amounts to showing
 \[f_!^{(\sV)}\cdot c\leq\psi^Y\cdot d\cdot Tf\iff f_!^{(\sW)}\cdot\varphi^X\cdot c\leq d\cdot Tf.\]
 But this is obvious: given the left-hand inequality, compose it from the left with $\varphi^Y$ and use 
 $\varphi^Y\cdot\psi^Y\leq 1_{\sW^Y}$ and 
 $f_!^{(\sW)}\cdot\varphi^X\leq\varphi^Y\cdot f_!^{(\sV)}$ to obtain the right-hand inequality. The converse direction is similar.
\end{proof}

If we are given lax distributive laws $\lambda,\kappa$ of $\bbT$ over $\bbP_{\sV},\bbP_{\sW}$, respectively, what it takes for $B_{\varphi}$ to map $(\lambda,\sV)\text{-}{\bf Alg}$ into $(\kappa,\sW)\text{-}{\bf Alg}$ is well known from the context of $(\bbT,\sV)$-categories (see \cite{MonTop, LaxDistLaws}): $\varphi: (\sV,\otimes,\sk)\to(\sW,\otimes,{\sf l})$ should be a {\em lax homomorphism} of quantales, that is: monotone, with ${\sf l}\leq\sk$ and $\varphi(u)\otimes\varphi(v)\leq\varphi(u\otimes v)$, for all $u, v\in\sV$; in addition, $\varphi$ should satisfy the $\lambda\text{-}\kappa${\em -compatibility condition} $ \varphi^{TX}\cdot\lambda_X\geq\kappa_X\cdot T(\varphi^X)$. 

However, in order to be able to restrict the adjunction of Lemma \ref{graphadjunction} to the categories of lax algebras, while $\psi$ needs to satisfy these conditions, no additional condition (beyond monotonicity) is required for its left adjoint $\varphi$, thanks to the following simple fact:

\begin{lem}\label{reflexivity}
For a lax distributive law $\lambda$ of $\bbT$ over $\bbP_{\sV}$, $(\lambda,\sV)\text{-}{\bf Alg}$ is reflective in $(\bbT,\sV)\text{-}{\bf Gph}$. The reflector assigns to a $(\bbT,\sV)$-graph $(X,c)$ the lax $(\lambda,\sV)$-algebra $(X,\overline{c})$,
with \[\overline{c}=\bw\{c':TX\to\sV^X\;|\;c\leq c', \,(X,c')\in (\lambda,\sV)\text{-}{\bf Alg}\}.\]
\end{lem}

\begin{proof}
As infima in $\sV^X$ are formed pointwise, and as $\lambda$ is monotone, when all $c'\geq c$ satisfy (R), (T), the same is true for $\overline{c}$, since
\[c'\cdot e_X\geq \sy_X\quad\text{ and  }\quad c'\cdot m_X\geq\sfs_X\cdot c'_!\cdot \lambda_X\geq\sfs_X\cdot(\overline{c})_!\cdot\lambda_X.\]
Furthermore, for any morphism $f:(X,c)\to(Y,d)$ one has $c\leq f^!\cdot d\cdot Tf$,
where $f_!\dashv f^!:\sV^Y\to\sV^X$. Since $c\leq f^!\cdot d\cdot Tf$  satisfies (R), (T) when $d$ does, in that case one has 
$\overline{c}\leq  f^!\cdot d\cdot Tf$, and therefore a morphism $f:(X,\overline{c})\to(Y,d)$.
\end{proof}

\begin{prop}\label{changeofbaseadjunction}
For lax distributive laws $\lambda, \kappa$ of a monad $\bbT$ over $\bbP_{\sV},\bbP_{\sW}$, respectively, and a lax homomorphism $\psi:\sW\to\sV$ 
that preserves infima and satisfies the $\lambda\text{-}\kappa$-compatibility condition, the change-of-base functor
\[B_{\psi}:(\kappa,\sW)\text{-}{\bf Alg}\to(\lambda,\sV)\text{-}{\bf Alg}\]
has a left adjoint $\overline{B}_{\varphi}$, given by $(X,c)\mapsto (X,\overline{\varphi^X\cdot c})$, where $\varphi\dashv \psi$.
\end{prop}

\begin{proof}
As an infima-preserving map of complete lattices, $\psi$ does indeed have left adjoint $\varphi$. Following Lemma \ref{graphadjunction} and \ref{reflexivity}, the left adjoint of the functor $B_{\psi}$ is just the composite of the two left adjoints established previously.
\end{proof}

We can now apply the Proposition to the ultrafilter law $\beta=\beta^{(\sV)}$ of Proposition \ref{Ufilterdistributes}, first noting that any map $\varphi:\sV\to\sW$ 
satisfies the $\beta^{(\sV)}\text{-}\beta^{(\sW)}$-compatibility condition--strictly so, as a quick inspection reveals.
As usual, we write $\sV\text{-}{\bf Cat}$ for $(\mathbb{I},\sV)\text{-}{\bf Cat}\cong (1,\sV)\text{-}{\bf Alg}$ (where $\mathbb I$ is the identical monad on {\bf Set} and $1:\sP_{\sV}\to\sP_{\sV}$ the idential transformation)). Recall that a quantale $\sV$ is {\em  integral} when its $\otimes$-neutral element $\sk$ is the top element $\top$ in $\sV$.

\begin{thm}\label{equivalences}
For completely distributive and integral quantales $\sV, \sW$, let $\varphi:\sV\to\sW$ be monotone and $\psi:\sW\to\sV$ a lax homomorphism of quantales. Then the following statements are equivalent:
\begin{enumerate}
\item[{\rm (i)}] $\varphi\dashv\psi:\sW\to\sV$;
\item[{\rm (ii)}] $\overline{B}_\varphi\dashv B_\psi:(\beta,\sW)\text{-}\bf{Alg}\to(\beta,\sV)\text{-}\bf{Alg}$;
\item[{\rm (iii)}] $\overline{B}_\varphi\dashv B_\psi:\sW\text{-}\bf{Cat}\to\sV\text{-}\bf{Cat}$.
\end{enumerate}
\end{thm}

\begin{proof}
The implications (i)$\Longrightarrow$(ii) and (i)$\Longrightarrow$(iii)) follow from Proposition \ref{changeofbaseadjunction}.

For (iii)$\Longrightarrow$(i) we must show $v\leq\psi\varphi(v)$ and $\varphi\psi(w)\leq w$, 
for all $v\in\sV, w\in\sW$.
Consider $X=\{x,y\}$ with $x\neq y$ and, for any $v\in\sV$, define a $\sV$-category structure $a:X\nrightarrow X$ on $X$ by $a(x,x)=a(y,y)=\sk$ and $a(x,y)=a(y,x)=v$. Since $\sW$ is integral, one easily sees that the least $\sW$-category structure $b$ on $X$ with $\varphi a\leq b$ 
is given by $b(x,x)=b(y,y)=\top$ and $b(x,y)=b(y,x)=\varphi(v)$; hence, $\overline{B}_{\varphi}(X,a)=(X,b)$. Since the adjunction unit 
$(X,a)\to B_{\psi}\overline{B}_{\varphi}(X,a)=(X,\psi b)$ is a $\sV$-functor,
$v=a(x,y)\leq\psi b(x,y)=\psi\varphi(v)$ follows. Similarly one shows $\varphi\psi(w)\leq w$, and (i) follows.

For (ii)$\Longrightarrow$(i), one may proceed as in (iii)$\Longrightarrow$(i), simply because, for finite $X$, one has $\sU X\cong X$. 
\end{proof}
We note that the equivalence of (i) and (ii) appears in \cite{Fang2008}, Theorem 3.1, under the hypothesis that both $\varphi$ and $\psi$ be lax homomorphisms of quantales.

For the sake of completeness we also note that the hypothesis of Theorem \ref{equivalences} that $\psi$ be a lax homomorphism, comes for free when $\varphi$ is a
{\em homomorphism} of quantales, {\em i.e.}, a sup-preserving map which also preserves the monoid structure of the quantales, thanks to the following Proposition.

\begin{prop}\label{homomorphism}
When $\sW$ is integral, the right adjoint of a homomorphism $\varphi:\sV\to\sW$ of quantales is a lax homomorphism of quantales.
\end{prop}

\begin{proof}
Let $\psi:\sW\to\sV$ be the right adjoint of the sup-preserving map $\varphi$. Since $\psi$ preserves infima,  $\psi(\top)=\top\geq \sk$. 
Also, for all $u,w\in\sW$,
\[\psi(u\otimes w)\geq \psi(\varphi \psi(u)\otimes \varphi \psi(w))=
\psi\varphi(\psi (u)\otimes \psi(w))\geq \psi(u)\otimes \psi(w).\] 
\end{proof}

We will now apply Proposition \ref{changeofbaseadjunction} and Theorem \ref{equivalences} to the only homomorphism $\iota: {\sf 2}=\{\bot<\top\}\to\sV$, 
given by $\iota(\bot)=\bot,\,\iota(\top)=\sk$. The monotone map $\iota$ has a right adjoint $\pi$,
given by $(\pi(v)=\top\iff v\geq\sk)$ for all $v\in\sV$, which is a lax homomorphism of quantales. If $\sV$ is integral, $\iota$ has also a left adjoint $o$,
given by $(o(v)=\bot\iff v=\bot)$ for all $v\in\sV$. Considering the identical monad ${\mathbb I}$ one obtains the well-known fact (see \cite{MonTop}) that, for $\sV$ non-trivial,  $B_{\iota}$ embeds the category {\bf Ord} of (pre)ordered sets and monotone maps as a full coreflective subcategory into 
$\sV\text{-}{\bf Cat}$, which is also reflective when $\sV$ is integral. Complete distributivity of $\sV$ is not needed for this, but it becomes essential now when we consider the ultrafilter monad $\bbU$ and its lax distributive law $\beta$.

\begin{cor}
For $\sV$ completely distributive, $B_{\iota}$ embeds the category {\bf Top} of topological spaces into $\sV\text{-}{\bf App}$ as a full coreflective subcategory, with coreflector
$B_{\pi}$. If $\sV$ is integral, the embedding is also reflective, with reflector $\overline{B}_{o}$. In particular, 
{\bf Top} is both, reflective and coreflective, in {\bf App}, as well as in {\bf ProbApp}.
\end{cor}

In fact, using the same technique as above we can refine the last statement of the Corollary and show:

\begin{cor}\label{AppinProbApp}
{\bf App} is fully embedded into {\bf ProbApp} as a reflective and coreflective subcategory.
\end{cor}

\begin{proof}
The homomorphism $\sigma: [0,\infty]\to{\bf {\Delta}}$ defined in Section 2 (after Corollary \ref{approach}), has
a right adjoint 
\[\rho:{\bf {\Delta}}\to [0,\infty],\;\varphi\mapsto \inf\{\alpha\in[0,\infty]\;|\;\varphi(\alpha)=1\},\]
which is a lax homomorphism of quantales, as well as a left adjoint
 \[\lambda:{\bf {\Delta}}\to [0,\infty], \;\varphi\mapsto\sup\{\alpha\in[0,\infty]\;|\;\varphi(\alpha)=0\}.\]
 One therefore has the adjunctions
 \[\overline{B}_{\lambda}\dashv B_{\sigma}\dashv B_{\rho}:
 (\beta,{\bf{\Delta}})\text{-}{\bf Alg}\to(\beta,[0,\infty])\text{-}{\bf Alg}.\]
\end{proof}

\begin{rem}
The proof of Corollary \ref{AppinProbApp} remains valid if we equip the set of distance distribution functions with the
monoidal structure
\[(\varphi\otimes\psi)(\gamma)={\rm sup}_{\alpha+\beta\leq\gamma}\varphi(\alpha)\&\psi(\beta),\]
where $\&$ is any left-continuous continuous t-norms on $[0,1]$, other than the ordinary multiplication which we used
to define the convolution product $\odot$ of ${\bf{\Delta}}$. In the case $\&=\wedge$, the corresponding proof was first carried out by J\"ager \cite{JagerApp}.
\end{rem}

Here is a third application of Theorem \ref{equivalences}:
\begin{exmp} 
Let $\mathsf{Dn}\sV$ be the set of all down-closed subsets of $\sV$ which, when ordered by inclusion, is a completely distributive lattice. It becomes a quantale with 
\[A\odot B=\{c\in\sV\;|\;\exists a\in A, b\in B: c\leq a\otimes b\}\quad(A,B\in\mathsf{Dn}\sV)\]
and $\odot$-neutral element the down-closure $\da \sk$ of the $\otimes$-neutral  element $\sk$ of $\sV$. 
$\mathsf{Dn}\sV$ is integral if, and only if, $\sV$ is integral. More importantly, if $\sV$ is completely distributive, we have adjunctions
\[\Da\ \dashv\sup\dashv\ \da:\sV\to\mathsf{Dn}\sV\]
(see \cite{Wood}). Furthermore, $\sup:\mathsf{Dn}\sV\to\sV$ is a homomorphism of quantales, while $\da$ is a lax homomorphism (but never a homomorphism, as it fails to preserve the bottom element). Therefore, we obtain the adjunctions
\[\overline{B}_{\Da}\dashv B_{\sup}\dashv B_{\da}:
 \sV\text{-}{\bf Cat}\to \mathsf{Dn}\sV\text{-}{\bf Cat}
 \quad\text{   and   }\quad\overline{B}_{\Da}\dashv B_{\sup}\dashv B_{\da}:
 (\beta,\sV)\text{-}\bf{Alg}\to(\beta,\mathsf{Dn}\sV)\text{-}\bf{Alg}.\]
 For $\sV$ an $n$-element chain, $\mathsf{Dn}\sV$ is an $(n+1)$-element chain, which contains two distinct copies of $\sV$, one reflectively embedded, the other coreflectively.  If $n>2$, $\sV$-categories are generalized (pre)ordered sets $X$, for which the truth value for two points in $X$ being related allows for a discrete linear range, beyond $\top$ or $\bot$. For $\sV=([0,\infty],\geq)$, in addition to the order embedding
 \[\da:[0,\infty]\to\mathsf{Dn}[0,\infty],\;\alpha\mapsto [\alpha,\infty],\]
 which preserves the monoidal structure, but is not a sup-map, one has the order embedding
 \[\curlyvee:[0,\infty]\to\mathsf{Dn}\sV,\;\alpha\mapsto(\alpha,\infty],\]
 which is a sup-map, but does not preserve the monoidal structure. Since $\mathsf{Dn}[0,\infty]$ is a disjoint union of
 the images of the two order embeddings, a $\mathsf{Dn}[0,\infty]$-category structure on a set $X$ will return to a pair of points in $X$ one of two types of distances, with one type always ranking below the other, despite having equal numerical value (since
 $(\alpha,\infty]\subset[\alpha,\infty]$, for all $\alpha\in[0,\infty]).$
  \end{exmp}


\end{document}